\documentclass[12pt]{article}
\usepackage[latin1]{inputenc}
\usepackage{amsmath,amsthm,amssymb}
\usepackage{amsfonts}
\usepackage{amsmath,amsthm,amssymb,amscd}
\usepackage{latexsym}
\usepackage{color}
\usepackage{color,enumitem,graphicx}
\usepackage[colorlinks=true,urlcolor=blue,
citecolor=red,linkcolor=blue,linktocpage,pdfpagelabels,
bookmarksnumbered,bookmarksopen]{hyperref}
\usepackage{graphicx}
\usepackage{mathrsfs}
\usepackage{bm}
\usepackage{cite}

\makeatletter \@addtoreset{equation}{section} \makeatother

\newtheorem{theorem}{Theorem}[section]

\newtheorem{proposition}{Proposition}[section]
\newtheorem{lemma}{Lemma}[section]
\newtheorem{remark}{Remark}[section]

\numberwithin{equation}{section}

\allowdisplaybreaks

\begin{document}

\title{Normalized solutions for a fractional $N/s$-Laplacian Choquard equation
with exponential critical nonlinearities}
\author
{Wenjing Chen\footnote{Corresponding author.}\ \footnote{E-mail address:\, {\tt wjchen@swu.edu.cn} (W. Chen),    {\tt zxwangmath@163.com} (Z. Wang)}\ \ and    Zexi Wang\\
\footnotesize  School of Mathematics and Statistics, Southwest University,
Chongqing, 400715, P.R. China}
\date{ }
\maketitle

\begin{abstract}
{In this paper, we are concerned with the following fractional $N/s$-Laplacian Choquard equation
 \begin{align*}
   \begin{cases}
   (-\Delta)^s_{N/s}u=\lambda |u|^{\frac{N}{s}-2}u +(I_\mu*F(u))f(u),\ \  \mbox{in}\ \mathbb{R}^N,\\
   \displaystyle\int_{\mathbb{R}^N}|u|^{N/s} \mathrm{d}x=a^{N/s},
   \end{cases}
   \end{align*}
where $s\in(0,1)$, $a>0$ is a prescribed constant, $\lambda\in \mathbb{R}$, $I_\mu(x)=\frac{1}{|x|^{\mu}}$ with $\mu\in(0,N)$, $F$ is the primitive function of $f$, and $f$ is a continuous function with exponential critical growth of Trudinger-Moser type. Under some suitable assumptions on $f$, we prove that the above problem admits a radial solution for any given $a>0$, by using the mountain-pass argument. 
}

\end{abstract}

\section{{\bfseries Introduction}}\label{introduction}
This paper is devoted to the following fractional $N/s$-Laplacian Choquard equation
\begin{align}\label{pro1}
   (-\Delta)^s_{N/s}u=\lambda |u|^{\frac{N}{s}-2}u +(I_\mu*F(u))f(u),\ \  \mbox{in}\ \mathbb{R}^N,
   \end{align}
where $s\in(0,1)$, $\lambda\in \mathbb{R}$, $I_\mu(x)=\frac{1}{|x|^\mu}$ with $\mu\in(0,N)$, $F$ is the primitive function of $f$, and $f$ is a continuous function with exponential critical growth.
$(-\Delta)^s_p$ is the fractional $p$-Laplacian operator defined by
\begin{align*}
(-\Delta )^s_pu(x):=C(N,s)\lim\limits_{\varepsilon\rightarrow 0^+}\int_{\mathbb{R}^N\backslash B_\varepsilon(x)}\frac{|u(x)-u(y)|^{p-2}(u(x)-u(y))}{|x-y|^{N+ps}}\mathrm{d}y, \quad \text {in $\mathbb{R}^N$},
\end{align*}
for $u\in C_0^\infty(\mathbb{R}^N)$,
where $B_\varepsilon(x)$ denotes the ball of radius $\varepsilon$ around $x$ in $\mathbb{R}^N$. 
For further considerations about the fractional $p$-Laplacian operator,
we refer the readers to \cite{di}.

We emphasize that the fractional $p$-Laplacian is nonlinear and nonlocal, so they bring additional difficulties. Another motivation to investigate \eqref{pro1} comes from the Choquard nonlinearity involving exponential critical growth. For $s\in(0,1)$, if $N>ps$, we know the classical Sobolev embedding that $W^{s,p}(\mathbb{R}^N)\hookrightarrow L^q(\mathbb{R}^N)$ is continuous for $q\in[p,p_s^*]$, where $p_s^*:=Np/({N-ps})$ is called the Sobolev critical exponent. However, if $N=ps$, then $p_s^*=\infty$ and $W^{s,N/s}(\mathbb{R}^N)$ is not continuously embedded in $L^\infty(\mathbb{R}^N)$, see \cite{di} for more details. In the case $N=ps$, the maximum growth that allows us to treat \eqref{pro1} variationally in the Sobolev space $W^{s,N/s}(\mathbb{R}^N)$, which is motivated by
the fractional Trudinger-Moser inequality first given by Ozawa \cite{ozawa} and later improved by Kozono et al. \cite{kozono} (see also \cite{Sk,I,Z}). More precisely, inspired by \cite{kozono}, we say $f(t)$ has exponential critical growth if there exists $\alpha_0>0$ such that
\begin{align*}
\lim\limits_{|t|\to+\infty}\frac{|f(t)|}{e^{\alpha |t|^{\frac{N}{N-s}}}}=
\begin{cases}
0,\quad  &\mbox{for $\alpha>\alpha_0$,}\\
+\infty,\quad &\mbox{for $0<\alpha <\alpha_0$.}
\end{cases}
\end{align*}


To get solutions of $(\ref{pro1})$, one way is to fix $\lambda\in \mathbb{R}$ and look for 
 critical points of the energy functional $ \widetilde{\mathcal{J}}:W^{s,N/s}(\mathbb{R}^N)\rightarrow \mathbb{R}$ (see e.g. \cite{BM1,de,LY1,YRCW})
\begin{align*}
\widetilde{ \mathcal{J}}(u)=&\frac{s}{N} \int_{\mathbb{R}^N}\int_{\mathbb{R}^N}\frac{|u(x)-u(y)|^{N/s}}{|x-y|^{2N}}\mathrm{d}x\mathrm{d}y
-\frac{s\lambda}{N}\int_{\mathbb{R}^N} |u|^{N/s}\mathrm{d}x-\frac{1}{2} \int_{\mathbb{R}^N}(I_\mu*F(u))F(u)\mathrm{d}x,
\end{align*}
where 
$W^{s,N/s}(\mathbb{R}^N)$ is defined by
\begin{equation*}
  W^{s,N/s}(\mathbb{R}^N)=\big\{u\in L^{N/s}(\mathbb{R}^N):[u]_{s,N/s}<\infty\big\},
\end{equation*}
here $[u]_{s,N/s}$ denotes the Gagliardo seminorm, that is
\begin{equation*}
  [u]_{s,N/s}=\bigg(\int_{\mathbb{R}^N}\int_{\mathbb{R}^N}\frac{|u(x)-u(y)|^{N/s}}{|x-y|^{2N}}\mathrm{d}x\mathrm{d}y\bigg)^{s/{N}}.
\end{equation*}
It is well known that the space $\big(W^{s,N/s}(\mathbb{R}^N),\|\cdot\|_{W^{s,N/s}}\big)$, where $\|\cdot\|_{W^{s,N/s}}^{N/s}=[\cdot]_{s,N/s}^{N/s}+\|\cdot\|_{N/s}^{N/s}$, is a uniformly
convex Banach space, particularly reflexive, and separable (see \cite{PXZ}). We also remind the readers that $C_0^\infty(\mathbb{R}^N)$ is dense in $W^{s,N/s}(\mathbb{R}^N)$ (see \cite{AF}).

Alternatively, from a physical point of view, it is interesting to find solutions of $(\ref{pro1})$ having prescribed mass
 \begin{align}\label{pro2}
   \displaystyle\int_{\mathbb{R}^N}|u|^{N/s} \mathrm{d}x=a^{N/s}, \quad \text{with $a>0$}.
   \end{align}
In this direction, the parameter $\lambda\in \mathbb{R}$ appears as a Lagrange multiplier, which depends on the solution and is not a priori given.  
This type of solution is called the normalized solution, and can be obtained by
 looking for critical points of the following energy functional
\begin{align*}
\mathcal{ J}(u)=\frac{s}{N} \int_{\mathbb{R}^N}\int_{\mathbb{R}^N}\frac{|u(x)-u(y)|^{N/s}}{|x-y|^{2N}}\mathrm{d}x\mathrm{d}y-\frac{1}{2} \int_{\mathbb{R}^N}(I_\mu*F(u))F(u)\mathrm{d}x
\end{align*}
on the $L^{N/s}$-sphere $$ S(a)=\Big\{u\in W^{s,N/s}(\mathbb{R}^N):\int_{\mathbb{R}^N}|u|^{N/s} \mathrm{d}x=a^{N/s}\Big\}.$$


In the past few years, many researchers have studied normalized solutions for nonlinear elliptic equations, see \cite{alvesji,CSW,soave1,jeanjean,jeanjeanlu,soave,soave3} for Laplacian equations, \cite{DW,LQL,WLZL,ZZ2} for $p$-Laplacian equations.
Considering the following nonlinear Schr\"{o}dinger equation with a $L^2$-constraint
\begin{align}\label{jj}
  \begin{split}
  \left\{
  \begin{array}{ll}
   -\Delta u=\lambda u+f(u),
   \ \  \mbox{in}\ \mathbb{R}^N,\\
    \displaystyle\int_{\mathbb{R}^N}|u|^2\mathrm{d}x=a^2.
    \end{array}
    \right.
  \end{split}
  \end{align}
If $f(u)=|u|^{p-2}u$, $p\in (2,2^*]$ with $2^*=\frac{2N}{N-2}$ if $N\geq 3$, $2^*=+\infty$ if $N=1,2$, by using the Gagliardo-Nirenberg inequality \cite{Niren}, a new critical $2+\frac{4}{N}$ which is called $L^2$-critical exponent appears. In this case, the associated energy functional of \eqref{jj} is defined by
  \begin{equation*}
    \widehat{\mathcal{J}}(u)=\frac{1}{2}\int_{\mathbb{R}^N}|\nabla u|^2\mathrm{d}x-\frac{1}{p}\int_{\mathbb{R}^N}|u|^p\mathrm{d}x.
  \end{equation*}
  If $2<p<2+\frac{4}{N}$ ($L^2$-subcritical), then $\widehat{\mathcal{J}}$ is bounded from below on $ \widehat{S}(a):=\big\{u\in H^1(\mathbb{R}^N):\int_{\mathbb{R}^N}|u|^2 \mathrm{d}x=a^2\big\}$, so we can try to find a global minimizer of $\widehat{\mathcal{J}}|_{\widehat{S}(a)}$ as a solution of \eqref{jj}, see e.g. \cite{CL,Shibata}. While, if $2+\frac{4}{N}< p\leq2^*$ ($L^2$-supercritical), $\widehat{\mathcal{J}}$ is unbounded from below on $\widehat{S}(a)$, so it seems impossible to search for a global minimizer to obtain a solution.
  Jeanjean \cite{jeanjean} first obtained a normalized solution of \eqref{jj} when $f$ has a $L^2$-supercritical growth.
By using a minimax principle,
Bartsch and Soave \cite{soave1} presented a new approach, which is based on a natural constraint, and also proved the existence of normalized solutions in this case. Inspired by \cite{jeanjean,soave1}, Soave \cite{soave} studied \eqref{jj} with $f(u)=\omega|u|^{q-2}u+|u|^{r-2}u$, $\omega\in \mathbb{R}$, $2<q\leq 2+4/N\leq r<2^*$ and $q<r$.
Existence and asymptotic properties of normalized solutions, as well as stability and instability results were established in \cite{soave}. The Sobolev critical case $r=2^*$ with $N\geq3$ was considered by Soave \cite{soave3}. In the case $N=2$ and $f$ has exponential critical growth, the existence of normalized solutions of \eqref{jj} has been discussed by Alves et al. \cite{alvesji}.
Besides, Chen et al. \cite{CSW} considered normalized solutions of Schr\"{o}dinger equations with exponential critical nonlinearities in $\mathbb{R}$.


For $p$-Laplacian equations, \cite{ZZ2} is the first paper to study normalized solutions of $p$-Laplacian equations with a $L^p$-constraint
\begin{align}\label{Zhi}
   \begin{cases}
   -\Delta_pu=\lambda |u|^{p-2}u +\omega |u|^{q-2}u+f(u),\ \  \mbox{in}\ \mathbb{R}^N,\\
  \displaystyle\int_{\mathbb{R}^N}|u|^p \mathrm{d}x=a^p,
  \end{cases}
  \end{align}
  where $1<p<N$, $p<q<\bar{p}:=p+p^2/{N}$ ($L^p$-critical exponent), $\omega\in \mathbb{R}$, and $f\in C(\mathbb{R},\mathbb{R})$ is $L^p$-supercritical. Under some suitable assumptions on $\omega,q$ and $f$, the authors established the existence of least energy solutions, the multiplicity of radial solutions and nonradial sign-changing solutions of \eqref{Zhi}.
Furthermore, Deng and Wu \cite{DW} considered \eqref{Zhi}
with $f(u)=|u|^{p^*-2}u$, $p<q<p^*:=Np/(N-p)$. Using the Ekeland variational principle and minimax argument, they obtained several existence results and asymptotic behaviours. Moreover, the multiplicity of radial solutions was also established by genus theory when $p<q<\bar{p}$.

Motivated by the results already mentioned above, especially \cite{alvesji,DW}, a natural question is whether the normalized solutions exist for problem \eqref{pro1} -\eqref{pro2}.  To reach the conclusion, we
give the following assumptions.

\quad Assume that $f$ satisfies:\\
$(f_1)$   $f\in C(\mathbb{R},\mathbb{R})$, and $\lim\limits_{t\to0}\frac{|f(t)|}{|t|^\kappa}=0$ for some $\kappa>\frac{3N-2s-\mu}{2s}$;\\
$(f_2)$  $f$ has exponential critical growth at infinity, i.e., there exists $\alpha_0>0$ such that
\begin{align*}
\lim\limits_{|t|\to+\infty}\frac{|f(t)|}{e^{\alpha |t|^{\frac{N}{N-s}}}}=
\begin{cases}
0,\quad  &\mbox{for $\alpha>\alpha_0$,}\\
+\infty,\quad &\mbox{for $0<\alpha <\alpha_0$;}
\end{cases}
\end{align*}
$(f_3)$ there exists a constant $\theta>\frac{3N-\mu}{2s}$ 
 such that $0<\theta F(t)\leq tf(t)$ for any
 $t\neq 0$;\\
$(f_4)$ there exist $\tau>0$ and $\sigma>0$ such that $F(t)\geq \frac{\tau}{\sigma}|t|^\sigma$ for any $t\in \mathbb{R}$.





Our main result can be stated as follows:

\begin{theorem}\label{thm1.1}
Assume that  $(f_1)$-$(f_4)$ hold, then there exists $\tau_*>0$ such that for any $\tau\geq \tau_*$, problem $(\ref{pro1})$-\eqref{pro2} has a radial solution.
\end{theorem}

\begin{remark}
{\rm A typical example satisfying $(f_1)$-$(f_4)$ is
\begin{align*}
f(t)=\tau|t|^{\sigma-2}te^{\alpha_0 |t|^{\frac{N}{N-s}}}
\end{align*}
for any $\sigma>\max\{\kappa+1,\theta\}$}.
\end{remark}

    This paper is organized as follows. Section \ref{sec preliminaries} contains some preliminaries. In Section \ref{vf}, we give the variational framework of problem \eqref{pro1}-\eqref{pro2}. 
In Section \ref{minimax},  we use the mountain-pass argument to construct a bounded $(PS)$ sequence. 
 Finally, in Section \ref{proof}, we give the proof of Theorem \ref{thm1.1}. Throughout this paper, we will use the notation $\|\cdot\|_q:=\|\cdot\|_{L^q(\mathbb{R}^N)}$, $q\in [1,\infty]$, $C,C_i,i\in \mathbb{N}^+$ denote positive constants possibly different from line to line.


\section{{\bfseries Preliminaries}}\label{sec preliminaries}

We start with some preliminaries. 
\begin{proposition}\label{hardy}
\cite[Theorem 4.3]{Lieb} Let $1<r,t<\infty$ and $0<\mu<N$ with $\frac{1}{r}+\frac{1}{t}+\frac{\mu}{N}=2$. If $f\in L^r(\mathbb{R}^N)$ and $h\in L^t(\mathbb{R}^N)$, then there exists a sharp constant $C(N,\mu,r,t)>0$ such that
\begin{align}\label{HLS}
\int_{\mathbb{R}^N}\int_{\mathbb{R}^N}\frac{f(x)h(y)}{|x-y|^\mu}\mathrm{d}x\mathrm{d}y\leq C(N,\mu,r,t)\|f\|_r\|h\|_t.
\end{align}
\end{proposition}



\begin{lemma}\label{GN}\cite[Theorems 1.1-1.2]{Z1}
Let  $0<s<1$, $1<N/s<r<\infty$, $u\in W^{s,N/s}(\mathbb{R}^N)$, then there exists a sharp constant $C(N,s,r)>0$ such that
\begin{align}\label{gns}
\|u\|_r^r\leq C(N,s,r)[u]_{s,N/s}^{r-\frac{N}{s}}\|u\|_{N/s}^{N/s}.
\end{align}
\end{lemma}


\begin{lemma}\label{tm}
(i) \cite{kozono} For any $\alpha>0$ and $u\in W^{s,N/s}(\mathbb{R}^N)$, we have
\begin{equation*}
  \int_{\mathbb{R}^N}\Psi(\alpha |u(x)|^{N/(N-s)})\mathrm{d}x<+\infty;
\end{equation*}
(ii) \cite[Theorem 1.7]{Sk} Let $0<s<1<N/s$, then there exists $\alpha_*>0$ such that
\begin{align*}
  \sup_{u\in W^{s,N/s}(\mathbb{R}^N)\backslash \{0\},[u]_{s,N/s} \leq1}\frac{1}{\|u\|_{N/s}^{N/s}}\int_{\mathbb{R}^N}\Psi(\alpha |u(x)|^{N/(N-s)})\mathrm{d}x
  \begin{cases}
  <\infty, \quad \alpha<\alpha_*,\\
  =\infty, \quad \alpha\geq \alpha_*,
  \end{cases}
 \end{align*}
 where $\Psi(t)=e^t-\sum\limits_{j=0}^{j_{N,s}-2}\frac{t^j}{j!}$ and $j_{N,s}=\min\big\{j\in \mathbb{N}^+:j\geq N/s\big\}$.
\end{lemma}


\begin{lemma} \cite[Lemma 4.8]{Kavian}\label{weakcon}
Let $\Omega \subseteq \mathbb{R}^N$ be any open set. For $1<t<\infty$, let $\{u_n\}$ be bounded in $L^t(\Omega)$ and $u_n(x)\to u(x)$ a.e. in $\Omega$. Then $u_n(x) \rightharpoonup u(x)$ in $L^t(\Omega)$.
\end{lemma}

  \section{{\bfseries Variational framework}}\label{vf}
First of all, to make the notation concise, for $\alpha>\alpha_0$ and $t\in \mathbb{R}$, we set
\begin{equation*}
  R(\alpha,t)=e^{\alpha |t|^{\frac{N}{N-s}}}-\sum\limits_{j=0}^{j_{N,s}-2}\frac{\alpha^j}{j!}|t|^{\frac{Nj}{N-s}}=\sum\limits^{\infty}_{j_{N,s}-1}\frac{\alpha^j}{j!}|t|^{\frac{Nj}{N-s}},
\end{equation*}
where $j_{N,s}=\min\{j\in \mathbb{N}^+:j\geq N/s\}$.

  By using assumptions $(f_1)$ and $(f_2)$, it follows that for any $\zeta>0$, $q>\frac{3N}{2s}$ and $\alpha>\alpha_0$, there exists $C>0$ such that
\begin{align*}
|f(t)|\leq \zeta |t|^{\kappa}+C|t|^{q-1}R(\alpha,t), \quad \mbox{for all $t\in\mathbb{R}$,}
\end{align*}
and using $(f_3)$, we have
\begin{align}\label{Ft}
|F(t)|\leq \zeta |t|^{\kappa+1}+C|t|^{q} R(\alpha,t), \quad \mbox{for all $t\in\mathbb{R}$.}
\end{align}
Using \eqref{HLS}, \eqref{Ft} and Lemma \ref{tm}, we know $\mathcal{J}$ is well defined in $W^{s,N/s}(\mathbb{R}^N)$ and of class $C^1$ with
\begin{align*}
\langle \mathcal{J}'(u),v\rangle=&
\int_{\mathbb{R}^N}\int_{\mathbb{R}^N}\frac{|u(x)-u(y)|^{\frac{N}{s}-2}[u(x)-u(y)][v(x)-v(y)]}{|x-y|^{2N}}\mathrm{d}x\mathrm{d}y-\int_{\mathbb{R}^N}(I_\mu*F(u))f(u)v \mathrm{d}x,
\end{align*}
for any $u,v\in W^{s,N/s}(\mathbb{R}^N)$. Hence, a critical point of $\mathcal{J}|_{S(a)}$
corresponds to a solution of problem \eqref{pro1}-\eqref{pro2}.

To understand the geometry of $\mathcal{J}|_{S(a)}$, for any $\beta\in \mathbb{R}$ and $u\in W^{s,N/s}(\mathbb{R}^N)$, we define
\begin{equation*}
  \mathcal{H}(u,\beta)(x):=e^{s\beta}u(e^\beta x),\quad \text{for a.e. $x\in \mathbb{R}^N$}.
\end{equation*}
One can easily check that $\| \mathcal{H}(u,\beta)\|_{N/s}=\|u\|_{N/s}$ for any $\beta\in \mathbb{R}$. Thus, we can investigate the structure of $\mathcal{J}(\mathcal{H}(u,\beta))$ to speculate the structure of $\mathcal{J}|_{S(a)}$. Denote $S_r(a)=S(a)\cap W_{rad}^{s,N/s}(\mathbb{R}^N)$, where $W_{rad}^{s,N/s}(\mathbb{R}^N)$ is the subset of the radially symmetric functions in $W^{s,N/s}(\mathbb{R}^N)$.

 \begin{lemma}\label{strong1}
Assume that $(f_1)$-$(f_3)$ hold. Let $\{u_n\}\subset S_r(a)$ be a sequence 
 and satisfy $\sup\limits_{n\in\mathbb{N}^+}[u_n]_{s,N/s}^{N/s}<\big[\frac{(2N-\mu)\alpha_*}{2N \alpha_0}\big]^{(N-s)/{s}}$. If $u_n\rightharpoonup u$ in $W_{rad}^{s,N/s}(\mathbb{R}^N)$, then there exists $\alpha>\alpha_0$ close to $\alpha_0$ such that for all $q>1$, there holds
\begin{equation*}
 \int_{\mathbb{R}^N}|u_n|^{q}R(\alpha,u_n)\mathrm{d}x\rightarrow  \int_{\mathbb{R}^N}|u|^{q}R(\alpha,u)\mathrm{d}x,\,\,\,\text{as}\,\,\, n\rightarrow\infty.
\end{equation*}
\end{lemma}
\begin{proof}
First, we claim that there exist $\alpha>\alpha_0$ close to $\alpha_0$, $\nu>1$ close to $1$ such that $R(\alpha,u_n)$ is uniformly bounded in $L^\nu(\mathbb{R}^N)$. Indeed, fix $\alpha>\alpha_0$ close to $\alpha_0$,  there exists $\nu>1$ close to $1$ such that
\begin{equation*}
   \nu\alpha\sup\limits_{n\in\mathbb{N}^+}[u_n]_{s,N/s}^{N/(N-s)}<\alpha_*.
 \end{equation*}
Hence,  by Lemma \ref{tm}, we have%
\begin{equation*}
  \int_{\mathbb{R}^N}R(\alpha,u_n)^\nu\mathrm{d}x=\int_{\mathbb{R}^N}R\Big(\alpha[u_n]_{s,N/s}^{N/(N-s)},\frac{u_n}{[u_n]_{s,N/s}}\Big)^\nu\mathrm{d}x\leq \int_{\mathbb{R}^N}R\Big(\nu\alpha[u_n]_{s,N/s}^{N/(N-s)},\frac{u_n}{[u_n]_{s,N/s}}\Big)\mathrm{d}x\leq C,
 \end{equation*}
 this proves the claim.
 Since $u_n\rightharpoonup u$ in $W_{rad}^{s,N/s}(\mathbb{R}^N)$, then $u_n\rightarrow u$ a.e. in $\mathbb{R}^N$. By Lemma \ref{weakcon}, we obtain
 $R(\alpha,u_n)\rightharpoonup R(\alpha,u)$ in $L^\nu(\mathbb{R}^N)$. Moreover, for $\nu'=\frac{\nu}{\nu-1}$, using the compact  embedding $W_{rad}^{s,N/s}(\mathbb{R}^N)\hookrightarrow L^{q\nu'}(\mathbb{R}^N)$, we derive that $u_n\rightarrow u$ in $L^{q\nu'}(\mathbb{R}^N)$, and so
 $|u_n|^{q}\rightarrow |u|^{q}$ in $L^{\nu'}(\mathbb{R}^N)$ as $n\rightarrow\infty$.
Thus, by the definition of weak convergence, using the H\"{o}lder inequality, we infer that
\begin{align*}
  &\Big|\int_{\mathbb{R}^N}|u_n|^{q}R(\alpha,u_n)\mathrm{d}x-\int_{\mathbb{R}^N}|u|^{q}R(\alpha,u)\mathrm{d}x\Big|\\ \leq &\int_{\mathbb{R}^N}\big||u_n|^{q}-|u|^{q}\big|R(\alpha,u_n)\mathrm{d}x+\int_{\mathbb{R}^N}|u|^{q}|R(\alpha,u_n)-R(\alpha,u)|\mathrm{d}x\\
\leq&  \Big(\int_{\mathbb{R}^N}\big||u_n|^{q}-|u|^{q}\big|^{\nu'}\mathrm{d}x\Big)^{\frac{1}{\nu'}}\Big(\int_{\mathbb{R}^N}R(\alpha,u_n)^\nu\mathrm{d}x\Big)^{\frac{1}{\nu}}+\int_{\mathbb{R}^N}|u|^{q}|R(\alpha,u_n)-R(\alpha,u)|\mathrm{d}x\rightarrow 0,
\end{align*}
as $n\rightarrow\infty$.
\end{proof}

\begin{lemma}\label{strong2}
Assume that $(f_1)$-$(f_3)$ hold. Let $\{u_n\}\subset S_r(a)$ be a sequence 
 and satisfy $\sup\limits_{n\in\mathbb{N}^+}[u_n]_{s,N/s}^{N/s}<\big[\frac{(N-\mu)\alpha_*}{N \alpha_0}\big]^{(N-s)/{s}}$. If $u_n\rightharpoonup u$ in $W_{rad}^{s,N/s}(\mathbb{R}^N)$,
  then
\begin{equation*}
  \int_{\mathbb{R}^N}(I_\mu*F(u_n))F(u_n)\mathrm{d}x\rightarrow \int_{\mathbb{R}^N}(I_\mu*F(u))F(u)\mathrm{d}x,\,\,\,\text{as}\,\,\, n\rightarrow\infty,
\end{equation*}
\begin{equation*}
  \int_{\mathbb{R}^N}(I_\mu*F(u_n))f(u_n)u_n\mathrm{d}x\rightarrow \int_{\mathbb{R}^N}(I_\mu*F(u))f(u)u\mathrm{d}x,\,\,\,\text{as}\,\,\, n\rightarrow\infty,
\end{equation*}
and
\begin{equation*}
  \int_{\mathbb{R}^N}(I_\mu*F(u_n))f(u_n)\phi \mathrm{d}x\rightarrow \int_{\mathbb{R}^N}(I_\mu*F(u))f(u)\phi \mathrm{d}x,\,\,\,\text{as}\,\,\, n\rightarrow\infty
\end{equation*}
for any $\phi\in C_0^\infty(\mathbb{R}^N)$.
\end{lemma}

\begin{proof}
First, we claim that
$I_\mu\ast F(u_n)$ belongs to $L^\infty(\mathbb{R}^N)$. Indeed, by \eqref{Ft}, we have
\begin{align*}
  &\big|I_\mu\ast F(u_n)\big|\\ \leq& \int_{\mathbb{R}^N}\frac{\zeta|u_n|^{\kappa+1}+C|u_n|^{q}R(\alpha,u_n)}{|x-y|^\mu}dy\\
  \leq&
  C\int_{|x-y|\leq1}\frac{|u_n|^{\kappa+1}+|u_n|^{q}R(\alpha,u_n)}{|x-y|^\mu}dy+
  C\int_{|x-y|\geq1}\Big(\frac{|u_n|^{\kappa+1}}{|x-y|^\mu}+|u_n|^{q+1}R(\alpha,u_n)\Big)dy\\
  =&C_1+C\Big(\int_{|x-y|\leq1}\frac{|u_n|^{\kappa+1}}{|x-y|^\mu}dy+\int_{|x-y|\leq1}\frac{|u_n|^{q+1}R(\alpha,u_n)}{|x-y|^\mu}dy+\int_{|x-y|\geq1}|u_n|^{q+1}
 R(\alpha,u_n)dy\Big)\\
  =&:C_1+C(I+II+III).
\end{align*}
Choosing $\alpha>\alpha_0$ close to $\alpha_0$, $t>\frac{N}{N-\mu}$ close to $\frac{N}{N-\mu}$ and $\nu>1$ close to $1$ such that
$$\alpha t\nu\sup\limits_{n\in \mathbb{N}^+}[u_n]_{s,N/s}^{N/(N-s)}< \alpha_*.$$ Then for $t'=\frac{t}{t-1}$ and $\nu'=\frac{\nu}{\nu-1}$,
by $\kappa>\frac{3N-2s-\mu}{2s}>\frac{N-s}{s}$, using Lemma \ref{tm}, the H\"{o}lder inequality and Sobolev inequality, we have
\begin{equation*}
  I\leq \Big(\int_{|x-y|\leq1}|u_n|^{(\kappa+1)t}dy\Big)^{\frac{1}{t}}\Big(\int_{|x-y|\leq1}\frac{1}{|x-y|^{\mu t'}}dy\Big)^{\frac{1}{t'}}\leq C_2,
\end{equation*}
\begin{align*}
  II&\leq \Big(\int_{|x-y|\leq1}|u_n|^{qt}R\Big(\alpha t[u_n]_{s,N/s}^{N/(N-s)},\frac{u_n}{[u_n]_{s,N/s}}\Big)dy\Big)^{\frac{1}{t}}\Big(\int_{|x-y|\leq1}\frac{1}{|x-y|^{\mu t'}}dy\Big)^{\frac{1}{t'}}\\
  &\leq C\Big(\int_{|x-y|\leq1}|u_n|^{qt\nu'}dy\Big)^{\frac{1}{t\nu'}}
  \Big(\int_{|x-y|\leq1}R\Big(\alpha t\nu[u_n]_{s,N/s}^{N/(N-s)},\frac{u_n}{[u_n]_{s,N/s}}\Big)dy\Big)^{\frac{1}{t\nu}}
  \leq C_3,
\end{align*}
and
\begin{equation*}
  III\leq \Big(\int_{|x-y|\geq1}|u_n|^{(q+1)\nu'}\Big)^{\frac{1}{\nu'}} \Big(\int_{|x-y|\geq1}R\Big(\alpha \nu[u_n]_{s,N/s}^{N/(N-s)},\frac{u_n}{[u_n]_{s,N/s}}\Big)dy\Big)^{\frac{1}{\nu}}\leq C_4.
\end{equation*}
This proves the claim.

Hence
\begin{equation*}
  (I_\mu*F(u_n))F(u_n)\rightarrow (I_\mu*F(u))F(u)\quad \text{a.e. in $\mathbb{R}^N$},
\end{equation*}
\begin{equation*}
 |(I_\mu*F(u_n))F(u_n)|\leq C|F(u_n)|\leq C|u_n|^{\kappa+1}+C|u_n|^{q}R(\alpha,u_n)\quad\text{for all $n\in \mathbb{N}^+$},
\end{equation*}
and
\begin{equation*}
  |u_n|^{\kappa+1}+|u_n|^{q}R(\alpha,u_n)\rightarrow |u|^{\kappa+1}+|u|^{q}R(\alpha,u)\quad\text{a.e. in $\mathbb{R}^N$}.
\end{equation*}
By Lemma \ref{strong1}
,  the compact embedding $W^{s,N/s}_{rad}(\mathbb{R}^N)\hookrightarrow L^{\kappa+1}(\mathbb{R}^N)$,
and the Lebesgue dominated convergence theorem, we get
\begin{equation*}
  \int_{\mathbb{R}^N}(I_\mu*F(u_n))F(u_n)\mathrm{d}x\rightarrow \int_{\mathbb{R}^N}(I_\mu*F(u))F(u)\mathrm{d}x,\,\,\,\text{as}\,\,\, n\rightarrow\infty.
\end{equation*}
Similarly, we can prove that others hold. 
\end{proof}

\section{{\bfseries The minimax approach}}\label{minimax}
Denote $\widetilde{\mathcal{J}}(u,\beta)=\mathcal{J}(\mathcal{H}(u,\beta))$, we will prove that $\widetilde{\mathcal{J}}$ possesses a kind of mountain-pass geometrical structure.


\begin{lemma}\label{mountain}
Assume that $(f_1)$-$(f_3)$ hold. Let $u\in S_r(a)$ be arbitrary but fixed, then we have\\
(i)  $\widetilde{\mathcal{J}}(u,\beta)\to0^+$ as $\beta\to -\infty$; \\
(ii)  $\widetilde{\mathcal{J}}(u,\beta)\to -\infty$ as $\beta\to +\infty$.
\end{lemma}

\begin{proof}
$(i)$ By a straightforward calculation, we have
\begin{align*}
\int_{\mathbb{R}^N}|\mathcal{H}(u,\beta) |^{N/s} \mathrm{d}x=a^{N/s},\ \ \int_{\mathbb{R}^N}|\mathcal{H}(u,\beta) |^\xi \mathrm{d}x=e^{(s \xi-N)\beta}\int_{\mathbb{R}^N}|u|^\xi\mathrm{d}x, \ \ \text{for any $\xi>N/s$},
\end{align*}
and
\begin{align*}
\int_{\mathbb{R}^N}\int_{\mathbb{R}^N}\frac{|\mathcal{H}(u,\beta)(x)-\mathcal{H}(u,\beta)(y)|^{N/s}}{|x-y|^{2N}}
\mathrm{d}x\mathrm{d}y
=e^{N\beta}\int_{\mathbb{R}^N}\int_{\mathbb{R}^N}\frac{| u(x)-u(y)|^{N/s}}{|x-y|^{2N}}
\mathrm{d}x\mathrm{d}y.
\end{align*}
Thus there exist $\beta_1<<0$  such that
\begin{equation*}
\int_{\mathbb{R}^N}\int_{\mathbb{R}^N}\frac{|\mathcal{H}(u,\beta)(x)-\mathcal{H}(u,\beta)(y)|^{N/s}}{|x-y|^{2N}}
\mathrm{d}x\mathrm{d}y<\Big[\frac{(2N-\mu)\alpha_*}{2N \alpha_0}\Big]^{\frac{N-s}{s}},\quad \text{for any $\beta<\beta_1$.}
\end{equation*}
  Fix $\alpha>\alpha_0$ close to $\alpha_0$ and $\nu>1$ close to $1$ such that
\begin{equation*}
  \frac{2N\alpha \nu}{2N-\mu}[\mathcal{H}(u,\beta)]_{s,N/s}^{N/(N-s)}<\alpha_*,\quad \text{for any $\beta<\beta_1$}.
\end{equation*}
 Then, for $\frac{1}{\nu}+\frac{1}{\nu'}=1$, using \eqref{HLS}, $(\ref{Ft})$, Lemma \ref{tm}, the H\"{o}lder and Sobolev inequality, we have
\begin{align}\label{tain1}
&\int_{\mathbb{R}^N}\big(I_\mu*F(\mathcal{H}(u,\beta) )\big)F(\mathcal{H}(u,\beta) )\mathrm{d}x\leq \|F(\mathcal{H}(u,\beta))\|_{\frac{2N}{2N-\mu}}^2\nonumber\\
\leq&
\zeta \|\mathcal{H}(u,\beta) \|_{\frac{2N(\kappa+1)}{2N-\mu}}^{2(\kappa+1)}+C \bigg[\int_{\mathbb{R}^N} \big[R(\alpha,\mathcal{H}(u,\beta)) |\mathcal{H}(u,\beta)|^{q}\big]^{\frac{2N}{2N-\mu}} \mathrm{d}x\bigg]^{\frac{2N-\mu}{N}}\nonumber\\
\leq &\zeta \|\mathcal{H}(u,\beta) \|_{\frac{2N(\kappa+1)}{2N-\mu}}^{2(\kappa+1)}+C \bigg[\int_{\mathbb{R}^N} \Big[R\Big(\alpha [\mathcal{H}(u,\beta)]_{s,N/s}^{N/(N-s)},\frac{\mathcal{H}(u,\beta)}{[\mathcal{H}(u,\beta)]_{s,N/s}}\Big) |\mathcal{H}(u,\beta)|^{q}\Big]^{\frac{2N}{2N-\mu}} \mathrm{d}x\bigg]^{\frac{2N-\mu}{N}}\nonumber\\
\leq&\zeta \|\mathcal{H}(u,\beta) \|_{\frac{2N(\kappa+1)}{2N-\mu}}^{2(\kappa+1)}+C
\bigg[\int_{\mathbb{R}^N}\Big[R\Big(\alpha [\mathcal{H}(u,\beta)]_{s,N/s}^{N/(N-s)},\frac{\mathcal{H}(u,\beta)}{[\mathcal{H}(u,\beta)]_{s,N/s}}\Big)\Big]^{\frac{2N\nu}{2N-\mu}}\mathrm{d}x\bigg]^{\frac{2N-\mu}{N\nu}}\| \mathcal{H}(u,\beta) \|_{\frac{2Nq\nu'}{2N-\mu}}^{2q}\nonumber\\
\leq&\zeta \|\mathcal{H}(u,\beta) \|_{\frac{2N(\kappa+1)}{2N-\mu}}^{2(\kappa+1)}+C
\bigg[\int_{\mathbb{R}^N}R\Big(\frac{2N\alpha \nu}{2N-\mu} [\mathcal{H}(u,\beta)]_{s,N/s}^{N/(N-s)},\frac{\mathcal{H}(u,\beta)}{[\mathcal{H}(u,\beta)]_{s,N/s}}\Big)\mathrm{d}x\bigg]^{\frac{2N-\mu}{N\nu}}\| \mathcal{H}(u,\beta) \|_{\frac{2Nq\nu'}{2N-\mu}}^{2q}\nonumber\\
\leq& \zeta \|\mathcal{H}(u,\beta) \|_{\frac{2N(\kappa+1)}{2N-\mu}}^{2(\kappa+1)}+C\| \mathcal{H}(u,\beta) \|_{{\frac{2Nq\nu'}{2N-\mu}}}^{2q}\nonumber\\
=&\zeta e^{(2s\kappa+2s+\mu -2N)\beta}
\| u \|_{\frac{2N(\kappa+1)}{2N-\mu}}^{2(\kappa+1)}+Ce^{\frac{(2qs\nu'+\mu -2N)\beta}{\nu'}}
\| u\|_{{\frac{2Nq\nu'}{2N-\mu}}}^{2q}.
\end{align}
Since $\kappa>\frac{3N-2s-\mu}{2s}$, $q>\frac{3N}{2s}$ and $\nu'=\frac{\nu}{\nu-1}$ large enough, we have
\begin{align*}
 \widetilde{\mathcal{J}}(u,\beta)\geq &\frac{s}{N}e^{N\beta}\int_{\mathbb{R}^N}\int_{\mathbb{R}^N}\frac{| u(x)-u(y)|^{N/s}}{|x-y|^{2N}}
\mathrm{d}x\mathrm{d}y\\&- C e^{(2s\kappa+2s+\mu -2N)\beta}
\| u \|_{\frac{2N(\kappa+1)}{2N-\mu}}^{2(\kappa+1)}-Ce^{\frac{(2qs\nu'+\mu -2N)\beta}{\nu'}}
\| u\|_{{\frac{2Nq\nu'}{2N-\mu}}}^{2q} \rightarrow0^+,\ \ \mbox{as} \ \beta\to-\infty,
\end{align*}
and by $(f_3)$
\begin{align*}
 \widetilde{\mathcal{J}}(u,\beta)\leq &\frac{s}{N}e^{N\beta}\int_{\mathbb{R}^N}\int_{\mathbb{R}^N}\frac{| u(x)-u(y)|^{N/s}}{|x-y|^{2N}}
\mathrm{d}x\mathrm{d}y\rightarrow 0^+,\ \ \mbox{as} \ \beta\to-\infty.
\end{align*}
So we have $\widetilde{\mathcal{J}}(u,\beta)\to0^+$ as $\beta\to -\infty$.

$(ii)$
For any fixed $\beta>>0$,  set
\begin{align*}
\mathcal{W}(t):=\int_{\mathbb{R}^N}(I_\mu*F(tu))F(tu)\mathrm{d}x,\quad \text{for any $t>0$}.
\end{align*}
Using $(f_3)$, one has
\begin{align*}
\frac{\frac{\mathrm{d}\mathcal{W}(t)}{\mathrm{d}t}}{\mathcal{W}(t)}>\frac{2\theta}{t},\quad \text{for any $t>0$}.
\end{align*}
Thus, integrating this over $\big[1,e^{s\beta}\big]$, we get
\begin{align}\label{fff}
\int_{\mathbb{R}^N}(I_\mu*F(e^{s\beta}u))F(e^{s\beta}u)\mathrm{d}x\geq e^{2s\theta\beta}\int_{\mathbb{R}^N}(I_\mu*F(u))F(u)\mathrm{d}x.
\end{align}
Hence,
\begin{align*}
\widetilde{\mathcal{J}}(u,\beta)\leq & \frac{s}{N}e^{N\beta}\int_{\mathbb{R}^N}\int_{\mathbb{R}^N}\frac{| u(x)-u(y)|^{N/s}}{|x-y|^{2N}}
\mathrm{d}x\mathrm{d}y- \frac{1}{2} e^{(2s\theta+\mu-2N)\beta}\int_{\mathbb{R}^N}(I_\mu*F(u))F(u)\mathrm{d}x.
\end{align*}
Since $\theta>\frac{3N-\mu}{2s}$, the above inequality yields that $\widetilde{\mathcal{J}}(u,\beta)\to -\infty$ as $\beta\to+\infty$.
\end{proof}

\begin{lemma}\label{minimax2}
Assume that $(f_1)$-$(f_3)$ hold. Then there exist $0<k_1<k_2$ such that
\begin{equation*}
  0<\inf\limits_{u\in \mathcal{A}} \mathcal J(u)\leq \sup\limits_{u\in \mathcal{A}} \mathcal J(u)<\inf\limits_{u\in \mathcal{B}} \mathcal J(u)
\end{equation*}
with
\begin{equation*}
  \mathcal{A}=\Big\{u\in S_r(a):[u]_{s,N/s}\leq k_1\Big\},\quad\mathcal{B}=\Big\{u\in S_r(a):[u]_{s,N/s}=k_2\Big\}.
\end{equation*}
\end{lemma}

\begin{proof}
If $k_2<\big[\frac{(2N-\mu)\alpha_*}{2N \alpha_0}\big]^{(N-s)/{N}}$, then for any $u\in \mathcal{B}$,
\begin{equation*}
[u]_{N/s}^{N/s}<\Big[\frac{(2N-\mu)\alpha_*}{2N \alpha_0}\Big]^{\frac{N-s}{s}}.
\end{equation*}
Fix $\alpha>\alpha_0$ close to $\alpha_0$ and $\nu>1$ close to $1$ such that
\begin{equation*}
  \frac{2N\alpha \nu}{2N-\mu}[u]_{s,N/s}^{N/({N-s})}<\alpha_*.
\end{equation*}
From \eqref{gns} and \eqref{tain1}, we obtain
\begin{align*}
\int_{\mathbb{R}^N}\big(I_\mu*F(u)\big)F(u)\mathrm{d}x \leq& C \|u \|_{\frac{2N(\kappa+1)}{2N-\mu}}^{2(\kappa+1)}+C\| u \|_{{\frac{2Nq\nu'}{2N-\mu}}}^{2q}
\leq C a^{\frac{2N-\mu}{s}} [u]_{s,N/s}^{2\kappa+2+\frac{\mu }{s}-\frac{2N}{s}} +C a^{\frac{2N-\mu}{s\nu'}} [u]_{s,N/s}^{\frac{2qs\nu'+\mu -2N}{s\nu'}}.
\end{align*}
Since $\kappa>\frac{3N-2s-\mu}{2s}$, $q>\frac{3N}{2s}$, and $\nu'=\frac{\nu}{\nu-1}$ large enough, 
for any $u\in \mathcal{B}$, there holds
\begin{equation*}
  \mathcal{J}(u)>\frac{s}{N}[u]_{s,N/s}^{N/s}-C a^{\frac{2N-\mu}{s}} [u]_{s,N/s}^{2\kappa+2+\frac{\mu }{s}-\frac{2N}{s}} +C a^{\frac{2N-\mu}{s\nu'}} [u]_{s,N/s}^{\frac{2qs\nu'+\mu -2N}{s\nu'}}>\widetilde{C}>0.
\end{equation*}
Similarly, we can prove that $\inf\limits_{u\in \mathcal{A}} \mathcal J(u)>\widehat{C}>0$.

On the other hand, by $(f_3)$, we have
\begin{equation*}
  \mathcal{J}(u)<\frac{s}{N}[u]_{s,N/s}^{N/s},
\end{equation*}
which implies that $\sup\limits_{u\in \mathcal{A}}\mathcal{J}(u)<\widetilde{C}$ for any $k_1\in(0,k_2)$ small enough. This ends the proof.
\end{proof}
Following by \cite{willem}, we recall that for any $a>0$, the tangent space of $S_r(a)$ at $u$ is defined by
\begin{align*}
T_u:=\Big\{v\in W_{rad}^{s,N/s}(\mathbb{R}^N) : \int_{\mathbb{R}^N}|u|^{\frac{N}{s}-2}uv\mathrm{d}x=0\Big\},
\end{align*}
and the tangent space of $S_r(a)\times \mathbb{R}$ at $(u,t)$ is defined by
\begin{align*}
\widetilde{T}_{u,t}:=\Big\{(v,k)\in W_{rad}^{s,N/s}(\mathbb{R}^N)\times \mathbb{R} : \int_{\mathbb{R}^N}|u|^{\frac{N}{s}-2}uv\mathrm{d}x=0\Big\}.
\end{align*}
\begin{lemma}\label{minimax3}
Let $k_1,k_2$ be defined in Lemma \ref{minimax2}. Then there exist $\hat{u},\tilde{u}\in S_r(a)$ such that

$(i)$ $[\hat{u}]_{s,N/s}\leq k_1$;

$(ii)$ $[\tilde{u}]_{s,N/s}> k_2$;

$(iii)$ $\mathcal J(\hat{u})>0>\mathcal J(\tilde{u})$.
\\Moreover, setting
\begin{equation*}
  \widetilde{m}_\tau(a)=\inf\limits_{\widetilde{h}\in \widetilde{\Gamma}_a}\max\limits _{t\in[0,1]}\widetilde{\mathcal J}({\widetilde{h}(t)})
\end{equation*}
with
\begin{equation*}
  \widetilde{\Gamma}_a=\big\{\widetilde{h}\in C([0,1],S_r(a)\times\mathbb{R} ):\widetilde{h}(0)=(\hat{u},0),\widetilde{h}(1)=(\tilde{u},0)\big\},
\end{equation*}
and
\begin{equation*}
  m_\tau(a)=\inf\limits_{h\in \Gamma_a}\max\limits _{t\in[0,1]}\mathcal J({h(t)})
\end{equation*}
with
\begin{equation*}
  \Gamma_a=\big\{h\in C([0,1],S_r(a)):h(0)=\hat{u},h(1)=\tilde{u}\big\},
\end{equation*}
then we have
\begin{equation*}
   \widetilde{m}_\tau(a)=m_\tau(a)\geq\max\{\mathcal J(\hat{u}),\mathcal J(\tilde{u})\}>0.
\end{equation*}
\end{lemma}

\begin{proof}
For any fixed $u_0\in S_r(a)$, by Lemmas \ref{mountain} and \ref{minimax2}, there exist two numbers $s_1<<-1$ and $s_2>>1$ such  that $\hat{u}=\mathcal{H}(u_0,s_1)$ and $\tilde{u}=\mathcal{H}(u_0,s_2)$ satisfy $(i)$-$(iii)$. For any $\widetilde{h}\in \widetilde{\Gamma}_a$, we write it into
\begin{equation*}
  \widetilde{h}(t)=(\widetilde{h}_1(t),\widetilde{h}_2(t))\in S(a)\times\mathbb{R}.
\end{equation*}
Setting $h(t)=\mathcal{H}(\widetilde{h}_1(t),\widetilde{h}_2(t))$, then $h(t)\in \Gamma_a$ and
\begin{equation*}
  \max\limits _{t\in[0,1]}\widetilde{\mathcal J}({\widetilde{h}(t)})=\max\limits _{t\in[0,1]}\mathcal J({h(t)})\geq  m_\tau(a).
\end{equation*}
By the arbitrariness of $\widetilde{h}\in \widetilde{\Gamma}_a$, we get $\widetilde{m}_\tau(a)\geq m_\tau(a)$.

On the other hand, for any $h\in \Gamma_a$, if we set $\widetilde{h}(t)=(h(t),0)$, then $\widetilde{h}(t)\in \widetilde{\Gamma}_a$ and
\begin{equation*}
  \widetilde{m}_\tau(a)\leq \max\limits _{t\in[0,1]}\widetilde{\mathcal J}({\widetilde{h}(t)})=\max\limits _{t\in[0,1]}\mathcal J({h(t)}).
\end{equation*}
By the arbitrariness of $h\in {\Gamma}_a$, we get $\widetilde{m}_\tau(a)\leq m_\tau(a)$. Hence, we have $\widetilde{m}_\tau(a)= m_\tau(a)$, and $m_\tau(a)\geq\max\{\mathcal J(\hat{u}),\mathcal J(\tilde{u})\}$ follows from the definition of $m_\tau(a)$.
\end{proof}

Learning from \cite[Proposition 2.2]{jeanjean}, by the standard Ekeland variational principle and pseudo-gradient flow, we have the following proposition, which gives the existence of the $(PS)_{\widetilde{m}_\tau(a)}$ sequence for $\widetilde{\mathcal{J}}(u,\beta)$ on $S_r(a)\times \mathbb{R}$.
\begin{proposition}\label{psse}
Let $\widetilde{h}_n\subset \widetilde{\Gamma}_a$ be such that
\begin{equation*}
  \max\limits _{t\in[0,1]}\widetilde{\mathcal J}({\widetilde{h}_n(t)})\leq \widetilde{m}_\tau(a)+\frac{1}{n}.
\end{equation*}
Then there exists a sequence $\{(v_n,\beta_n)\}\subset S_r(a)\times\mathbb{R} $ such that as $n\rightarrow\infty$,

$(i)$ $\widetilde{\mathcal J}(v_n,\beta_n)\rightarrow \widetilde{m}_\tau(a)$;

$(ii)$ $\widetilde{\mathcal J}'|_{S_r(a)\times\mathbb{R}}(v_n,\beta_n)\rightarrow0$, i.e.,
\begin{equation*}
  \partial_\beta\widetilde{\mathcal J}(v_n,\beta_n)\rightarrow0 \quad \text{and}\quad\langle\partial_v\widetilde{\mathcal J}(v_n,\beta_n),\widetilde{\varphi}\rangle\rightarrow0
\end{equation*}
for all
\begin{equation*}
  \widetilde{\varphi}\in T_{v_n,\beta_n}=\Big\{\widetilde{\varphi}=(\widetilde{\varphi}_1,\widetilde{\varphi}_2)\in W_{rad}^{s,N/s}(\mathbb{R}^N)\times \mathbb{R}:\int_{\mathbb{R}^N}|v_n|^{\frac{N}{s}-2}v_n\widetilde{\varphi}_1\,\mathrm{d}x=0\Big\}.
\end{equation*}
\end{proposition}
\begin{lemma}\label{pssequence}
For the sequence $\{(v_n,\beta_n)\}\subset S_r(a)\times\mathbb{R}$ obtained in Proposition \ref{psse}, setting $u_n=\mathcal{H}(v_n,\beta_n)$, then as $n\rightarrow\infty$, we have

$(i)$ ${\mathcal J}(u_n)\rightarrow {m}_\tau(a)$;

$(ii)$ $P(u_n)\rightarrow0$, where
  \begin{align*}
  P(u)=&\int_{\mathbb{R}^N}\int_{\mathbb{R}^N}\frac{|u(x)-u(y)|^{N/s}}{|x-y|^{2N}}\mathrm{d}x\mathrm{d}y+
  \frac{2N-\mu}{2s}\int_{\mathbb{R}^N}(I_\mu*F(u))F(u)\mathrm{d}x -\int_{\mathbb{R}^N}(I_\mu*F(u))f(u)u\mathrm{d}x;
  \end{align*}

$(iii)$ ${\mathcal J}'|_{S_r(a)}(u_n)\rightarrow0$, i.e.,
\begin{equation*}
  \langle{\mathcal J}'(u_n),{\varphi}\rangle\rightarrow0
\qquad
\mbox{for\ all}\ \ \ \
  {\varphi}\in T_{u_n}=\Big\{{\varphi}\in W_{rad}^{s,N/s}(\mathbb{R}^N):\int_{\mathbb{R}^N}|u_n|^{\frac{N}{s}-2}u_n{\varphi}\,\mathrm{d}x=0\Big\}.
\end{equation*}
\end{lemma}

\begin{proof}
For $(i)$, since ${\mathcal J}(u_n)=\widetilde{\mathcal J}(v_n,\beta_n)$ and ${m}_\tau(a)= \widetilde{m}_\tau(a)$, we get the conclusion.

For $(ii)$, we have
\begin{align*}
  \partial_\beta\widetilde{\mathcal J}(v_n,\beta_n)=&\partial_\beta\bigg[\frac{s}{N}e^{N\beta_n}[v_n]_{s,N/s}^{N/s}-\frac{e^{(\mu-2N)\beta_n}}{2}\int_{\mathbb{R}^N}(I_\mu*F(e^{s\beta_n}v_n))F(e^{s\beta_n}v_n)\mathrm{d}x\bigg]\\
=&se^{N\beta_n}[v_n]_{s,N/s}^{N/s}+\frac{2N-\mu}{2} e^{(\mu-2N)\beta_n} \int_{\mathbb{R}^N}  (I_\mu*F(e^{s\beta_n}v_n))F(e^{s\beta_n} v_n)\mathrm{d}x \\&- se^{(\mu-2N)\beta_n} \int_{\mathbb{R}^N}  (I_\mu*F(e^{s\beta_n}v_n))f(e^{s\beta_n} v_n)e^{s\beta_n} v_n\mathrm{d}x=sP(u_n).
\end{align*}
Thus $P(u_n)\rightarrow0$ as $n\rightarrow\infty$.

For $(iii)$, on the one hand, for any $\widetilde{\varphi}=(\widetilde{\varphi}_1,\widetilde{\varphi}_2)\in T_{v_n}$, we have
\begin{align*}
\langle\partial_v\widetilde{\mathcal J}(v_n,s_n),\widetilde{\varphi}\rangle=&e^{N\beta_n}\int_{\mathbb{R}^N}\int_{\mathbb{R}^N}\frac{|v_n(x)-v_n(y)|^{\frac{N}{s}-2}[v_n(x)-v_n(y)][\widetilde{\varphi}_1(x)-\widetilde{\varphi}_1(y)]}{|x-y|^{2N}}\mathrm{d}x\mathrm{d}y
\\&-e^{(\mu-2N)\beta_n}\int_{\mathbb{R}^N}(I_\mu*F(e^{s\beta_n}v_n))f(e^{s\beta_n}v_n)e^{s\beta_n}\widetilde{\varphi}_1\mathrm{d}x.
  \end{align*}
On the other hand,  \begin{align*}
                       \langle{\mathcal J}'(u_n),{\varphi}\rangle=&\int_{\mathbb{R}^N}\int_{\mathbb{R}^N}\frac{|u_n(x)-u_n(y)|^{\frac{N}{s}-2}[u_n(x)-u_n(y)][\varphi(x)-\varphi(y)]}{|x-y|^{2N}}\mathrm{d}x\mathrm{d}y
                       \\&-\int_{\mathbb{R}^N}(I_\mu*F(u_n))f(u_n)\varphi \mathrm{d}x\\
=& e^{(N-s)\beta_n}\int_{\mathbb{R}^N}\int_{\mathbb{R}^N}\frac{|v_n(x)-v_n(y)|^{\frac{N}{s}-2}[v_n(x)-v_n(y)][\varphi(e^{-\beta_n}x)-\varphi(e^{-\beta_n}y)]}{|x-y|^{2N}}\mathrm{d}x\mathrm{d}y
                     \\&-e^{(\mu-2N)\beta_n}\int_{\mathbb{R}^N}(I_\mu*F(e^{s\beta_n}v_n))f(e^{s\beta_n}v_n))\varphi(e^{-\beta_n}x) \mathrm{d}x.
                    \end{align*}
Taking $\varphi(e^{-\beta_n}x) =e^{s\beta_n}\widetilde{\varphi}_1$, then $\langle{\mathcal J}'(u_n),{\varphi}\rangle\rightarrow 0$ as $n\rightarrow\infty$, and $\varphi(x) =e^{s\beta_n}\widetilde{\varphi}_1(e^{\beta_n}x)$. If we can prove $\varphi \in T_{u_n}$, we get $(iii)$. In fact, it follows from the following equality:
\begin{equation*}
  \int_{\mathbb{R}^N}|u_n|^{\frac{N}{s}-2}u_n{\varphi}\mathrm{d}x=\int_{\mathbb{R}^N}|e^{s\beta_n}v_n(e^{\beta_n}x)|^{\frac{N}{s}-2}e^{s\beta_n}v_n(e^{\beta_n}x)e^{s\beta_n}\widetilde{\varphi}_1(e^{\beta_n}x)\mathrm{d}x= \int_{\mathbb{R}^N}|v_n|^{\frac{N}{s}-2}v_n\widetilde{{\varphi}}_1\mathrm{d}x=0.
\end{equation*}
\end{proof}

\begin{lemma}\label{energy}
Assume that $(f_1)$-$(f_4)$ hold, then we have $\lim\limits_{\tau\rightarrow+\infty}m_\tau(a)=0$.
\end{lemma}
\begin{proof}
For any fixed $u_0\in W^{s,N/s}(\mathbb{R}^N)$, as the proof of Lemma \ref{minimax3},  for any $t\in [0,1]$, $h_0(t)=\mathcal{H}(u_0,(1-t)s_1+ts_2)$ is a path in $\Gamma_a$. Hence, by $(f_4)$ and \eqref{fff}, we have
\begin{align*}
  m_\tau(a)\leq \max\limits_{t\in[0,1]}\mathcal J({h_0(t)})\leq& \max\limits_{\kappa>0}\Big\{\frac{s\kappa^N}{N}[u_0]_{s,N/s}^{N/s}
-\frac{\tau^2\kappa^{(2 s\theta+\mu+2\sigma-2N)}}{2\sigma^2}\int_{\mathbb{R}^N}(I_\mu*|u_0|^\sigma)|u_0|^\sigma \mathrm{d}x\Big\}\\=&C\Big(\frac{1}{\tau^2}\Big)^{\frac{N}{2 s\theta+\mu+2\sigma-3N}}.
\end{align*}
This together with $\theta>\frac{3N-\mu}{2s}$ and $\sigma>0$ yields $\lim\limits_{\tau\rightarrow\infty}m_\tau(a)=0$.
\end{proof}

For the sequence $\{u_n\}$ obtained in Lemma $\ref{pssequence}$, 
by the Lagrange multipliers rule, there exists a sequence $\{\lambda_n\}\subset \mathbb{R}$ such that
\begin{align}\label{key}
(-\Delta)^s_{N/s}u_n=\lambda_n|u_n|^{\frac{N}{s}-2}u_n+(I_\mu*F(u_n))f(u_n)+o_n(1).
\end{align}

\begin{lemma}\label{bdd}
Assume that $(f_1)$-$(f_3)$ hold, then $\{\lambda_{n}\}$ is bounded in $\mathbb{R}$. 
\end{lemma}
\begin{proof}
From $P(u_n)\rightarrow0$ and $\mathcal{J}(u_n)\rightarrow m_\tau(a)$ as $n\rightarrow\infty$,  using $(f_3)$, we have
\begin{align*}
  m_\tau(a)+o_n(1)=\mathcal{J}(u_n)-\frac{s}{N}P(u_n)\geq\frac{2s\theta+\mu -3N}{2N\theta} \int_{\mathbb{R}^N}(I_\mu*F(u_n))f(u_n)u_n\mathrm{d}x.
\end{align*}
Since $\theta>\frac{3N-\mu}{2s}$, up to a subsequence, we get
\begin{align}\label{1n}
\int_{\mathbb{R}^N}(I_\mu*F(u_n))f(u_n)u_n\mathrm{d}x\leq \frac{2N\theta m_\tau(a)}{2s\theta+\mu-3N}.
\end{align}
 Using $(f_3)$ and $P(u_n)\rightarrow0$ as $n\rightarrow\infty$ again, we know $\{u_n\}$ is bounded in $W_{rad}^{s,N/s}(\mathbb{R}^N)$. Up to a subsequence, we assume that $u_n\rightharpoonup u$ in $W_{rad}^{s,N/s}(\mathbb{R}^N)$.
Testing \eqref{key} with $u_n$, we have
\begin{equation}\label{bdd1}
  \lambda_{n}a^{N/s}=\int_{\mathbb{R}^N}\int_{\mathbb{R}^N}\frac{|u_n(x)-u_n(y)|^{N/s}}{|x-y|^{2N}}\mathrm{d}x\mathrm{d}y-\int_{\mathbb{R}^N}(I_\mu*F(u_n)) f({u_n})u_n \mathrm{d}x+o_n(1),
\end{equation}
which implies $\{\lambda_{n}\}$ is bounded in $\mathbb{R}$. Up to a subsequence, we assume that $\lambda_{n}\rightarrow\lambda$ as $n\rightarrow\infty$.
\end{proof}

\begin{lemma}\label{non}
Assume that $(f_1)$-$(f_4)$ hold, if
\begin{equation*}
m_\tau(a)<\frac{s}{N}\Big[\frac{(2N-\mu)\alpha_*}{2N \alpha_0}\Big]^{(N-s)/{s}},
\end{equation*}
then $u_n\rightharpoonup u\neq0$ in $W_{rad}^{s,N/s}(\mathbb{R}^N)$.
\end{lemma}

\begin{proof}
If $u=0$, then by Lemma \ref{strong2}, we obtain
\begin{align*}
\int_{\mathbb{R}^N}(I_\mu*F(u_n))F(u_n)\mathrm{d}x\rightarrow0,\quad \text{as $n\rightarrow\infty$}.
\end{align*}
So we can fix $\alpha>\alpha_0$ close to $\alpha_0$ and $\nu>1$ close to $1$ such that
\begin{equation*}
  \frac{2N\alpha \nu}{2N-\mu}\sup\limits_{n\in\mathbb{N}^+}[u_n]_{s,N/s}^{N/{N-s}}<\alpha_*.
\end{equation*}
From $(\ref{tain1})$, by $\kappa>\frac{3N-2s-\mu}{2s}$, $\nu'=\frac{\nu}{\nu-1}$ large enough and the compact embedding, we have
\begin{align*}
\|F(u_n)\|_{\frac{2N}{2N-\mu}}\leq
C\|u_n\|_{\frac{2N(\kappa+1)}{2N-\mu}}^{\kappa+1}+C
\|u_n\|_{\frac{2Nq\nu'}{2N-\mu}}^q\to0,\quad \mbox{as}\ n\to+\infty.
\end{align*}
By a similar argument as above, we infer that $ \|f(u_n)u_n\|_{\frac{2N}{2N-\mu}}\to0$ as $n\to\infty$.
Hence, we obtain
\begin{align*}
\int_{\mathbb{R}^N}(I_\mu*F(u_n))f(u_n)u_n\mathrm{d}x\rightarrow0,\quad \text{as $n\rightarrow\infty$}.
\end{align*}
 Using $P(u_n)\rightarrow 0$ again, we have $[u_n]_{s,N/s}\rightarrow0$ as $n\rightarrow\infty$, then $m_\tau(a)=0$, which is an absurd. Therefore, $u\neq0$.
\end{proof}

\begin{lemma}\label{ne}
Assume that $(f_1)$-$(f_4)$ hold, if
\begin{equation*}
m_\tau(a)<\frac{s}{N}\Big[\frac{(2N-\mu)\alpha_*}{2N \alpha_0}\Big]^{(N-s)/{s}},
\end{equation*}
then $\lambda<0$.
\end{lemma}

\begin{proof}
Combining \eqref{bdd1} with $P(u_n)\rightarrow0$ as $n\rightarrow\infty$, we get
\begin{align*}
\lambda_na^{N/s}=-\frac{2N-\mu}{2s}
\int_{\mathbb{R}^N}(I_\mu*F(u_n))F(u_n)\mathrm{d}x+o_n(1).
\end{align*}
Thanks to $u_n\rightharpoonup u\neq0$ in $W^{s,N/s}(\mathbb{R}^N)$, by $(f_3)$ and Fatou Lemma, we obtain
\begin{align*}
\limsup\limits_{n\rightarrow\infty}\lambda_n\leq -\frac{2N-\mu}{2sa^{N/s}}
\int_{\mathbb{R}^N}(I_\mu*F(u_n))F(u_n)\mathrm{d}x\leq -\frac{2N-\mu}{2sa^{N/s}}
\int_{\mathbb{R}^N}(I_\mu*F(u))F(u)\mathrm{d}x<0.
\end{align*}
Therefore, up to a subsequence, we can assume that $\lambda_n\to\lambda<0$ as $n\to\infty$.
\end{proof}




\section{{\bfseries Proof of the main result}}\label{proof}

\noindent{\bfseries Proof of Theorem \ref{thm1.1}.} Under the assumptions of Theorem \ref{thm1.1}, from \eqref{key}, \eqref{1n}, Lemmas \ref{strong2}, \ref{energy}, \ref{non} and \ref{ne},  we can see that there exists $\tau_*>0$ such that for any $\tau\geq \tau_*$, $u$ is a nontrivial weak solution of $(\ref{pro1})$-\eqref{pro2} with $\lambda<0$. Thus
\begin{equation}\label{con1}
  [u]_{s,N/s}^{N/s}=\lambda\|u\|_p^p+\int_{\mathbb{R}^N}(I_\mu*F(u))f(u)u\mathrm{d}x.
\end{equation}
From \eqref{key} and $\lambda_n\rightarrow\lambda$ as $n\rightarrow\infty$, we also have
\begin{equation}\label{con2}
  [u_n]_{s,N/s}^{N/s}=\lambda\|u_n\|_p^p+\int_{\mathbb{R}^N}(I_\mu*F(u_n))f(u_n)u_n\mathrm{d}x.
\end{equation}
Combining \eqref{con1}, \eqref{con2} with Lemma \ref{strong2}, we immediately obtain $u_n\rightarrow u$ in $W_{rad}^{s,N/s}(\mathbb{R}^N)$ as $n\rightarrow\infty$. This completes the proof.
    \qed

\end{document}